\documentclass[12pt,twoside,a4paper]{amsart}
\usepackage{amssymb}
\usepackage{graphicx}

\date{\today}
\usepackage[latin1]{inputenc}
\usepackage[T1]{fontenc}

\def\w{\wedge}

\def\dbar{\bar\partial}

\def\R{{\mathbb R}}
\def\C{{\mathbb C}}
\def\P{{\mathcal P}}

\def\B{{\mathcal B}}
\def\D{{d}}

\def\Hom{{\rm Hom\, }}

\def\Im{{\rm Im\, }}

\def\O{{\mathcal O}}

\def\Re{{\rm Re\,  }}

\def\U{{\mathcal U}}

\def\codim{\text{codim}\,}
\def\ann{\text{ann}\,}
\def\supp{\text{supp}\,}

\def\1{\mathbf 1}

\def\J{{\mathcal J}}

\def\be{\begin{equation}}
\def\ee{\end{equation}}

\def\J{{\mathcal J}}
\def\q{{\mathfrak q}}
\def\p{{\mathfrak p}}
\def\r{{\mathfrak r}}

\def\ass{{\text Ass}}

\def\ch{\mathcal{CH}}

\def\PM{\mathcal{PM}}

\newtheorem{thm}{Theorem}[section]
\newtheorem{lma}[thm]{Lemma}
\newtheorem{cor}[thm]{Corollary}
\newtheorem{prop}[thm]{Proposition}

\theoremstyle{definition}

\theoremstyle{remark}

\newtheorem{preremark}{Remark}
\newtheorem{preex}{Example}

\newenvironment{remark}{\begin{preremark}}{\qed\end{preremark}}
\newenvironment{ex}{\begin{preex}}{\qed\end{preex}}

\numberwithin{equation}{section}

\begin{document}

\title[Decomposition of residue currents]{Decomposition of residue currents}

\date{\today}

\author{Mats Andersson \& Elizabeth Wulcan}

\address{Mathematical Sciences, Chalmers University of Technology and Mathematical Sciences, University of Gothenburg \\S-412 96 G\"OTEBORG\\SWEDEN}

\email{matsa@chalmers.se,   wulcan@chalmers.se}



\begin{abstract}
Given a submodule $J\subset \O_0^{\oplus r}$ and a free resolution of $J$ one can define a certain vector-valued residue current whose annihilator is $J$. We make a decomposition of the current with respect to $\ass J$ that corresponds to a primary decomposition of $J$. As a tool we introduce a class of currents that includes usual residue and principal value currents; in particular these currents admit a certain type of restriction to analytic varieties and more generally to constructible sets.
\end{abstract}

%%%%%%%%%%

\maketitle

%\tableofcontents

\section{Introduction}\label{intro}
Let $(f_1,\ldots,f_q)$ be a holomorphic mapping at $0\in\C^n$ that forms a complete intersection, that is, the codimension of the common zero set $V^f=\{f_1=\cdots=f_q=0\}$ is equal to $q$. The Coleff-Herrera product
\begin{equation}\label{elfenben}
\mu^f=\dbar\frac{1}{f_1}\w\ldots\w\dbar\frac{1}{f_q},
\end{equation}
introduced in ~\cite{CH}, is a $\dbar$-closed $(0,q)$-current with support on $V^f$ such that $\bar \varphi \mu^f=0$ for all holomorphic $\varphi$ that vanish on $V^f$. It has turned out to be a good notion of a multivariate residue of $f$. The duality theorem, ~\cite{DS} and ~\cite{P2}, asserts that a holomorphic function $\varphi$ belongs to the ideal $J=(f_1,\ldots, f_q)$ in $\O_0$ if and only if the current $\varphi\mu^f$ vanishes, in other words the annihilator ideal $\ann\mu^f$ equals $J$.

Furthermore, $\mu^f$ has the so-called standard extension property, SEP, which basically means that $\mu^f$ has no ``mass'' concentrated on subvarieties of $V^f$ of codimension $>q$, or equivalently, its \emph{restriction}, in a sense that will be defined below, to each subvariety vanishes. Due to the SEP, $\mu^f$ can be decomposed in a natural way with respect to the irreducible components $V_j$ of $V^f$: $\mu^f=\sum_j \mu_j$, where $\mu_j$ is a current that has the SEP and whose support is contained in $V_j$; $\mu_j$ should be thought of as the restriction of $\mu^f$ to $V_j$.  
Moreover
\begin{equation}\label{eterisk}
J=\ann \mu^f=\cap_j \ann \mu_j. 
\end{equation}
From Proposition ~\ref{annanna} it follows that $\ann \mu_j$ is an $I_{V_j}$-primary ideal, where $I_{V_j}$ denotes the ideal associated with $V_j$, and hence \eqref{eterisk} gives a minimal primary decomposition of $J$. 
For a reference on primary decomposition see ~\cite{AM}. 
It is natural to consider the current $\mu^f$ as a geometric object and then $\mu^f=\sum_j \mu_j$ is a geometric decomposition of $\mu^f$.

\smallskip
In ~\cite{AW} we introduced, given a general ideal $J\subset\O_0$ a vector-valued residue current $R$ such that $\ann R=J$. The construction of $R$ is based on a free resolution of $J$ and it also involves a choice of Hermitian metrics on associated vector bundles (see Section ~\ref{stratifikation}). In case $J$ is defined by a complete intersection $f$, then $R$ is just the Coleff-Herrera product $\mu^f$. By means of the currents $R$ we were able to extend several results previously known for complete intersections. Combined with the framework of integral formulas developed in ~\cite{A4} we obtained explicit division formulas realizing the ideal membership, which were used to give for example a residue version of the Ehrenpreis-Palamodov fundamental principle, ~\cite{Ehr} and ~\cite{Pal}, generalizing ~\cite{BP}. 

\smallskip
In this paper we prove that the current $R$ can be decomposed as $R=\sum_\p R^\p$, where $\p$ runs over all associated prime ideals of $J$, so that $R^\p$ has support on $V(\p)$ and has the SEP. It is easy to see that this decomposition must be unique. Moreover it turns out that $\ann R^\p$ is $\p$-primary and that
\[J=\bigcap_\p \ann R^\p\] 
provides a minimal primary decomposition of $J$; our main result is Theorem ~\ref{prima}, which in fact holds also for submodules $J\subset \O_0^{\oplus r}$. 

As long as $J$ has no embedded primes the current $R^\p$ is just $R$ restricted to $V(\p)$ as for a complete intersection above, whereas the definition of $R^\p$ in general gets more involved. As a basic tool we introduce a class of currents that we call pseudomeromorphic and that admit restrictions to subvarieties and more generally to constructible sets. 
All currents in this paper are pseudomeromorphic and the definition is modeled on the currents that appear in various works as ~\cite{A}, ~\cite{AW} and ~\cite{PTY}; the typical example being the Coleff-Herrera product. This class has other desirable properties as well. It is closed under $\dbar$ and multiplication with smooth forms. If $T$ is pseudomeromorphic and has support on the variety $V$, then $T$ is annihilated by $\bar I_V$ and $\dbar \bar I_V$. In particular, (a version of) the SEP follows: if $T$ is of bidegree $(p,q)$ and $q<\codim V$, then $T$ vanishes. 
The relation $\varphi T=0$ is an intrinsic way of expressing that the result of the action of a list of differential operators applied to $\varphi$ vanishes on (certain subsets of) $V$. 
%That a function $\varphi$ annihilates $T$ is an intrinsic way of expressing that a certain family of differential operators applied to $\varphi$ vanishes on $V$. 
The fact that $\bar I_{V} T=0$ means that only holomorphic derivatives are involved. In case $T$ is $\dbar$-closed this can be made quite explicit, see ~\cite[Section 6]{Bj}.

\smallskip
In Section ~\ref{normaltyp} we define pseudomeromorphic currents, whereas restrictions to constructible sets are discussed in Section ~\ref{delsjon}. Section ~\ref{havet} deals with annihilators of pseudomeromorphic currents. Our main result, the decomposition of $R$, is presented in Section ~\ref{stratifikation} and a corresponding result in the algebraic case is given in Section ~\ref{algebraiska}. As an application we get a decomposition of the representation in our version of the fundamental principle.

\section{Pseudomeromorphic currents}\label{normaltyp}
Let $X$ be an $n$-dimensional complex manifold. Recall that the principal value current $\left[1/\sigma^a\right]$, $a$ positive integer, is well-defined in $\C_\sigma$, and that $\dbar\left[1/\sigma^a\right]$ is annihilated by $\bar\sigma$ and $d\bar\sigma$. In $\C^n_\sigma$, therefore, the current
\begin{equation}\label{principe}
\tau=\dbar \left [\frac{1}{\sigma_{i_1}^{a_{i_1}}}\right ]
 \wedge \ldots \wedge \dbar\left [\frac{1}{\sigma_{i_q}^{a_{i_q}}}\right ] \wedge 
\left [\frac{1}{\sigma_{i_{q+1}}^{a_{i_{q+1}}}}\right ]\cdots
\left [\frac{1}{\sigma_{i_\nu}^{a_{i_\nu}}}\right ] \alpha,
\end{equation}
where $\{i_1,\ldots,i_\nu\}\subset\{1,\ldots,n\}$, $a_k>0$, and $\alpha$ is a smooth form, is well-defined. 
If $\tau$ is a current on $X$, and there exists a local chart $\U_\sigma$ such that $\tau$ is of the form \eqref{principe} and $\alpha$ has compact support in $\U_\sigma$ we say that $\tau$ is \emph{elementary}. Note in particular that this definition, with $q$ equal to $0$, includes principal value currents as well as smooth forms.

A current $T$ on $X$ is said to be \emph{pseudomeromorphic} if it can be written as a locally finite sum 
\begin{equation}\label{guds}
T=\sum \Pi_* \tau_\ell,
\end{equation}
where each $\tau_\ell$ is a an elementary current on some manifold $\widetilde X_r$ and $\Pi=\Pi_1\circ\cdots\circ \Pi_r$ is a corresponding composition of resolutions of singularities $\Pi_1: \widetilde X_1\to X_1\subset X, \ldots, \Pi_r: \widetilde X_r\to X_r\subset \widetilde X_{r-1}$. We denote the class of pseudomeromorphic currents on $X$ by $\PM(X)$ and $\PM^{p,q}(X)$ denotes the elements that have bidegree $(p,q)$. 
Clearly the pseudomeromorphic currents is a subsheaf $\PM$ of the sheaf of all currents.

The Coleff-Herrera product \eqref{elfenben} and the more general products introduced in ~\cite{P} are typical examples of pseudomeromorphic currents. From the proof of Theorem 1.1 in ~\cite{PTY} and Theorem 1.1 in ~\cite{A} it follows that residue currents of Bochner-Martinelli type are pseudomeromorphic, and the arguments in Section ~2 in ~\cite{AW} shows that the residue currents introduced there are pseudomeromorphic.

Note that if $\tau$ is an elementary current, then $\dbar\tau$ is a sum of elementary currents and since $\dbar$ commutes with push-forwards it follows that $\PM$ is closed under $\dbar$. In the same way $\PM$ is closed under $\partial$. Moreover if $T$ is given by \eqref{guds} and $\beta$ is a smooth form, then 
$\beta\wedge T=\sum \Pi_*(\Pi^*\beta\wedge \tau_\ell)$, and thus $\PM$ is closed under multiplication with smooth forms. Furthermore $\PM$ admits a multiplication from the left with meromorphic currents: 

\begin{prop}\label{meromorf}
Let $T\in \PM$ and let $g$ be a holomorphic function. Then the analytic continuations 
\[
\left[\frac{1}{g}\right ] T:=\frac{|g|^{2\lambda}}{g}T\bigg |_{\lambda=0}
\quad 
\text{ and }
\quad
\dbar \left [\frac{1}{g}\right ] \wedge T:=\frac{\dbar|g|^{2\lambda}}{g}\wedge T\bigg |_{\lambda=0}
\]
exist and are pseudomeromorphic currents. The support of the second one is contained in $\{g=0\}\cap \supp T$. 
Moreover the products satisfy Leibniz' rule: 
\begin{equation}\label{regndag}
\dbar \left (
\left[\frac{1}{g}\right ] T \right )= 
\dbar \left[\frac{1}{g}\right ] \wedge T + \left[\frac{1}{g}\right ] \dbar T, 
\quad
\dbar \left (
\dbar \left[\frac{1}{g}\right ] \wedge T \right )= 
- \dbar \left[\frac{1}{g}\right ] \wedge \dbar T. 
\end{equation}
\end{prop}
By the first statement in the proposition we mean that the currents $(|g|^{2\lambda}/g)T$ and $(\dbar|g|^{2\lambda}/g)\wedge T$, which are clearly well defined if $\Re\lambda$ is large enough, have analytic continuations to $\Re\lambda >-\epsilon$ for some $\epsilon>0$, and $(|g|^{2\lambda}/g)T |_{\lambda=0}$ and $(\dbar|g|^{2\lambda}/g)\wedge T|_{\lambda=0}$ denote the values at $\lambda=0$. 

\begin{ex}\label{enkelprodukt}
In $\C$ the analytic continuations of $(|\sigma^a|^{2\lambda}/\sigma^a)\left[1 /\sigma^b\right]$, \linebreak $(|\sigma^a|^{2\lambda}/\sigma^a)\dbar \left[1 /\sigma^b\right]$ and $\dbar (|\sigma^a|^{2\lambda}/\sigma^a)\left[1 /\sigma^b\right]$ to $\Re\lambda > -\epsilon$ exist, which for instance can be seen by integration by parts, and we have
\begin{eqnarray*}
 \left[\frac{1}{\sigma^{a}}\right]\left[\frac{1}{\sigma^{b}}\right] &= &
|\sigma^a|^{2\lambda}\frac{1}{\sigma^a}\left[\frac{1}{\sigma^b}\right]\bigg |_{\lambda=0}=\left[\frac{1}{\sigma^{a+b}}\right]\\
 \left[\frac{1}{\sigma^{a}}\right]\dbar \left[\frac{1}{\sigma^{b}}\right] &=& 
|\sigma^a|^{2\lambda}\frac{1}{\sigma^a}\dbar\left[\frac{1}{\sigma^b}\right]\bigg|_{\lambda=0}=0\\
 \left (\dbar\left[\frac{1}{\sigma^{a}}\right]\right)\left[\frac{1}{\sigma^{b}}\right] &=&
\dbar |\sigma^a|^{2\lambda}\frac{1}{\sigma^a}\left[\frac{1}{\sigma^b}\right]\bigg|_{\lambda=0}=\dbar \left[\frac{1}{\sigma^{a+b}}\right].
\end{eqnarray*}
In particular it follows that the products with meromorphic currents in general are not (anti-)commutative. 
\end{ex}

\begin{proof}
Note that if $T$ is an elementary current and $g$ is a monomial, then, in light of Example ~\ref{enkelprodukt}, the analytic continuations exist and the values at $\lambda=0$ are elementary.

For the general case, assume that $T$ is of the form \eqref{guds}.
Locally, due to Hironaka's theorem on resolution of singularities, see ~\cite{BGVY}, for each $\ell$, in $\widetilde X_r$ we can find a resolution $\Pi^{r+1}: \widetilde X_{r+1}\to X_{r+1}\subset \widetilde X$ such that for each $k$, $(\Pi^{r+1})^* \sigma_k$ is a monomial times a nonvanishing factor and moreover $(\Pi^{r+1})^*(\Pi^r)^*\cdots (\Pi^1)^* g$ is a monomial. Thus we may assume that $\Pi^*g$ is a monomial for each $\ell$. 
Now, since $(|g|^{2\lambda}/g)T=\sum\Pi_*((|\Pi^*g|^{2\lambda}/\Pi^*g)\tau_\ell)$, the analytic continuation to $\Re\lambda > - \epsilon$ exists and the value at $\lambda=0$ is in $\PM$. 

The existence of the analytic continuation of $\dbar(|g|^{2\lambda}/g)\wedge T$ follows analogously. If $g\neq 0$ the value at $\lambda=0$ is clearly zero and hence the support of $\dbar[1/g]\wedge T$ is contained in $\{g=0\}\cap \supp T$. 

The last statement \eqref{regndag} follows directly from the definition and the uniqueness of analytic continuation. 
\end{proof}

If $T\in\PM(X)$ and $V\subset X$ is an analytic subvariety, we shall now see that the restriction $T|_U$ of $T$ to the Zariski-open set $U= V^c$ has a natural (standard) extension to $X$, which we denote $\1_U T$ (or, for typographical reasons, sometimes $T \1_U$). The current $T-\1_U T$, which has support on $V$, is a kind of residue that we will call the restriction of $T$ to $V$ and denote by $\1_V T$.

\begin{prop}\label{oppna}
Let $T\in\PM(X)$, let $U\subset X$ be a Zariski-open set, and assume that there is a tuple $h$ of holomorphic functions such that $\{h=0\}=U^c$. Then the analytic continuation $\1_U T:=|h|^{2\lambda}T|_{\lambda=0}$ exists and is independent of the particular choice of $h$.
\end{prop}

This gives a definition of $\1_UT$ for any Zariski-open set $U$ on any manifold. 
%The definition immediately extends to any Zariski-open set on any manifold. 

\begin{proof}
If $T$ is an elementary current \eqref{principe} and $h$ is a monomial the analytic continuation exists, compare to the proof of Proposition ~\ref{meromorf}, and it is easy to see that the value at $\lambda=0$ is $T$ if none of $\sigma_{i_1},\ldots,\sigma_{i_q}$ divide $h$ and zero otherwise. 

Assume that $T$ is of the form \eqref{guds}. Then, for each $\ell$ we can find resolutions of singularities $\Pi^{r+1}: \widetilde X_{r+1}\to X_{r+1}\subset \widetilde X_{r}$ and toric resolutions 
$\Pi^{r+2}: \widetilde X_{r+2}\to X_{r+2}\subset \widetilde X_{r+1}$ 
such that each $(\Pi^{r+2})^*(\Pi^{r+1})^* \sigma_k$ is a monomial times a nonvanishing factor and moreover \linebreak $(\Pi^{r+2})^*(\Pi^{r+1})^*(\Pi^r)^*\cdots (\Pi^1)^* h$ is a monomial $h^0$ times a nonvanishing tuple $h'$, see for example ~\cite{BGVY}. Thus in \eqref{guds} we may assume that each $\Pi^* h$ is a monomial times a nonvanishing tuple. Now, since $|h|^{2\lambda} T=\sum \Pi_* (\Pi^*|h|^{2\lambda} \tau_\ell)$, the analytic continuation to $\Re\lambda > -\epsilon$ exists. Moreover, 
\begin{equation}\label{gren}
|h|^{2\lambda} T|_{\lambda=0}=
\sum\Pi_*\tau_{\ell'},
\end{equation}
where the sum is taken over $\ell'$ such that 
none of the factors $\sigma_{i_1}, \ldots, \sigma_{i_q}$ in $\tau_{\ell'}$ divides $\Pi^* h$.
In particular it follows that $|h|^{2\lambda} T|_{\lambda=0}$ only depends on $U$ and not on the particular choice of $h$. Indeed, if $g$ is another tuple of functions such that $U^c=\{g=0\}$, we can find resolutions such that both $\Pi^* h$ and $\Pi^*g$ are monomials times nonvanishing tuples. Then clearly $\Pi^* h$ and $\Pi^*g$ must be divisible by the same coordinate functions.
\end{proof}

Let $T$ be a current on a variety $V$ of pure codimension $q$. Following Björk, see ~\cite{Bj} for background and a thorough discussion, we say that $T$ has the \emph{standard extension property (SEP) with respect to $V$} if the following holds: For each holomorphic $h$, not vanishing identically on any irreducible component of $V$, and each smooth approximand $\chi$ of the characteristic function of the interval $[1,\infty)$, 
\begin{equation}\label{limiten}
\lim_{\epsilon\to 0}\chi(|h|/\epsilon)T =T
\end{equation}
in the weak sense. As mentioned in the introduction the Coleff-Herrera product \eqref{elfenben} has the SEP with respect to $V^f$, see ~\cite{Bj}. Here the nontrivial case is when $\{h=0\}\supset V_{\text{sing}}$. 
In this case $\chi(|h|/\epsilon)\mu^f$ has meaning and \eqref{limiten} holds, even if $\chi$ is precisely equal to $\chi_{[1,\infty)}$.

%In this case one can give meaning to $\lim_{\epsilon\to 0}\chi(|h|/\epsilon)\mu^f$ so that \eqref{limiten} holds even if $\chi$ is not a smooth approximand, see ~\cite{Bj}. 

If $T$ is a pseudomeromorphic current one can verify, see for example ~\cite{A3}, that $\lim_{\epsilon\to 0}\chi(|h|/\epsilon)T=|h|^{2\lambda} T$. We will take as a definition that $T\in\PM$ (with support on $V$) has the SEP with respect to $V$ if $\1_{V'} T=0$ for all varieties $V'\subset V$ of codimension $\geq q+1$. 

The main use in ~\cite{Bj} of the notion of SEP is in the definition of (the sheaf of) Coleff-Herrera currents. A $(*, q)$-current $T$ with support on $V$ is a \emph{Coleff-Herrera current on $V$}, $T\in \ch_V$, if it has the SEP with respect to $V$, is $\dbar$-closed, and is annihilated by $\bar I_V$.

%Assume that $T$ is a pseudomeromorphic current with support on the analytic variety $V$ of (pure) codimension $k$. We say that $T$ has the \emph{standard extension property (SEP) with respect to $V$} if $\1_W T=0$ for all analytic varieties $W\subset V$ of codimension $>k$. Classically a current $T$ with support on $V$ has the SEP if it is equal to its own standard extension in the sense of Barlet ~\cite{B}, which means that $\lim_{\varepsilon \to 0}\chi (|h|/\varepsilon) T = T$ if $h$ is a holomorphic function that is generically nonvanishing on $V$, that is, $V\cap \{h=0\}$ has codimension $>k$, and $\chi$ is (a possibly smooth approximand of) the characteristic function for the interval $[1,\infty)$. One can show that $\lim_{\varepsilon \to 0}\chi (|h|/\varepsilon) T$ is indeed equal to $ T\1_{\{h=0\}^c} $, see ~\cite{A3} (in fact, this holds also for a tuple $h$ of holomorphic functions). Moreover, as we will see below, $T\1 _{\{h=0\}}=T \1_{V\cap \{h=0\}} $ if $T$ has support on $V$. Thus our definition of the SEP coincides with the classical one. 

\begin{prop}\label{antiholo}
Let $T\in\PM(X)$. Suppose that $\supp T$ is contained in the variety $Z$ and $\Psi$ is a holomorphic form that vanishes on $Z$. Then $\overline\Psi\wedge T=0$.
\end{prop}

\begin{proof}
Note that if $T$ is an elementary current and $Z$ is a union of coordinate hyperplanes the result follows from the one-dimensional case. Indeed, each term of $\Psi$ then contains a factor $\sigma_k$ or $d\sigma_k$ for each $\sigma_k$ that vanishes on $Z$, and moreover $\bar \sigma$ as well as $d\bar \sigma$ annihilate $\dbar [1/\sigma^a]$. 

For the general case assume that $T$ is given by \eqref{guds}. Note that $T=\1_ZT $ since $\supp T\subset Z$. The crucial point is now that according to the proof of Proposition ~\ref{oppna} we have $T=\sum\Pi_*\tau_{\ell'}$, where $\tau_{\ell'}$ is an elementary current with support on $(\Pi^L)^{-1}(Z)$, and hence
\[
\overline\Psi\wedge T
=
\sum\Pi_*\left ((\Pi)^* \overline\Psi \wedge \tau_{\ell'} \right ).
\]
Now, since $\Psi$ vanishes on $Z$, the holomorphic form $(\Pi)^* \Psi$ vanishes on $(\Pi)^{-1} (Z)$, which however is a union of coordinate planes. Hence $(\Pi)^* \overline\Psi \wedge \tau_{\ell'}$ vanishes as noted above and we are done.
\end{proof}

In particular, Proposition ~\ref{antiholo} implies that $d\overline h \wedge T=0$ if $h$ is holomorphic and vanishes on $\supp T$. Arguing as in the proofs of Theorems ~III.2.10-11 on normal currents in ~\cite{Dem} we get the following.

\begin{cor}\label{vaga}
Let $T\in\PM^{p,q}(X)$. If $\supp T$ is contained in the analytic variety $V$ of codimension $>q$, then $T=0$.
\end{cor}

In other words, the corollary says that if $T\in\PM^{p,q}(X)$ has support on $V$ of codimension $q$, then $T$ has the SEP. Also, Proposition ~\ref{antiholo} implies that $T$ is annihilated by all anti-holomorphic functions that vanish on $V$. Thus, if in addition $\dbar T=0$, then by definition $T\in\ch_V$.

Conversely, if $T\in\ch_V$, then locally $T=\gamma \wedge R$, where $R$ is a residue current and $\gamma$ is a holomorphic $(0,q)$-form, see for example ~\cite{A3}, and so $T\in\PM$. Hence we conclude:

\begin{prop}
Suppose that $V$ is an analytic variety of pure codimension $q$. 
Then $\ch_V^{p,q}$ is precisely the set of currents in $\PM^{p,q}$ with support on $V$ that are $\dbar$-closed. 
\end{prop}

By iterated use of Proposition \ref{meromorf}, given functions $f_1\ldots, f_\nu$, we can form a ``product''
\begin{equation}\label{arm}
T=\dbar \left [ \frac{1}{f_1}\right ] \wedge \ldots \wedge \dbar \left [ \frac{1}{f_q}\right ]\wedge 
\left [ \frac{1}{f_{q+1}}\right ] \ldots \left [ \frac{1}{f_\nu}\right ]\alpha,
\end{equation}
where $\alpha$ is a smooth form. If the $f_i$ are powers of coordinate functions (and $\alpha$ has compact support) we just get back \eqref{principe}. In general \eqref{arm} depends on the order of the $f_i$, compare to Example \ref{enkelprodukt}; to illustrate the usefulness of Corollary \ref{vaga} let us sketch a proof of the following claim:

\emph{
If $f_1,\ldots f_\nu$ form a complete intersection, then \eqref{arm} satisfies all formal (anti-)commutativity rules, and moreover 
\begin{equation}
f_1 T =0 \text{ and } 
f_\nu T = 
\dbar \left [ \frac{1}{f_1}\right ] \wedge \ldots \wedge \dbar \left [ \frac{1}{f_q}\right ]\wedge 
\left [ \frac{1}{f_{q+1}}\right ] \ldots \left [ \frac{1}{f_{\nu-1}}\right ]\alpha.
\end{equation}
}

In the complete intersection case \eqref{arm} coincides with the analogous product in \cite{P}. In particular, if $\nu=q$ and $\alpha\equiv 1$, then \eqref{arm} is the Coleff-Herrera product \eqref{elfenben}; compare to \cite{A3}. 
\begin{remark}
If $f_1, \ldots, f_\nu$ are arbitrary holomorphic functions one can check that \eqref{arm} coincides with the limit when $\epsilon_j\to 0$ of 
\begin{equation*}
\frac{\dbar\chi(|f_1|/\epsilon_1)\wedge \ldots \wedge \dbar \chi(|f_q|/\epsilon_q)\wedge\chi(|f_{q+1}|/\epsilon_{q+1}) \ldots \chi (|f_\nu|/\epsilon_\nu) \alpha}{f_1\ldots f_\nu},
\end{equation*}
where $\chi$ is as in \eqref{limiten}, 
provided that $\epsilon_1 >> \epsilon_2 >> \ldots >> \epsilon_\nu$ in the sense of \cite{CH}. 
\end{remark}
To prove the claim, assume for simplicity that $q=\nu=2$ and that $f$ and $g$ form a complete intersection; the general case is handled analogously. It follows from the definition that $f\dbar[1/f]=0$. Hence
\begin{equation}\label{lupp}
f\dbar \left [ \frac{1}{f}\right ] \wedge \left [ \frac{1}{g}\right ]\alpha =0 
\end{equation}
outside the set $\{f=g=0\}$, which by assumption has codimension 2, and by Corollary \ref{vaga} it then follows that \eqref{lupp} holds everywhere. In the same way one checks that $g\dbar [1/f] \wedge [1/g] \alpha = \dbar [1/f] \wedge \alpha$ and $[1/g]\dbar [1/f] \wedge \alpha = \dbar [1/f] \wedge [1/g]\alpha$. The remaining parts of the claim follow by Leibniz' rule after applying $\dbar$.

\section{Restrictions of pseudomeromorphic currents}\label{delsjon}
We will now show that one can give meaning to restrictions of pseudomeromorphic currents to all constructible sets. Recall that the set of constructible sets in $X$, which we will denote by $\mathcal C(X)$, is the Boolean algebra generated by the Zariski-open sets in $X$.

\begin{thm}\label{zariski}
There exists a unique, linear in $\PM(X)$ and degree-preserving, 
mapping 
\begin{equation}\label{kraka}
\mathcal C(X) \times \PM(X) \to \PM(X): (W,T) \mapsto \1_W T
\end{equation}
such that $\1_U T$ coincides with the natural extension across $U^c$ of the restriction $T|_U$ of $T$ to $U$ if $U\subset X$ is Zariski-open and moreover for all $W$ and $W'$ in $\mathcal C(X)$,
\begin{enumerate}
\item[(i)]
$\1_{W^c} T=T- \1_W T$ 
\item[(ii)]
$\1_{W\cap W'} T =\1_W (\1_{W'} T)$.
\end{enumerate}
\end{thm}

The uniqueness of \eqref{kraka} follows from (i) and (ii), since any constructible set can be obtained from a finite number of Zariski-open sets by taking intersections and complements.

If $X'$ is a open subset of $X$ and $T|_{X'}$ is the natural restriction of $T\in\PM(X)$ to $X'$, then Theorem \ref{zariski} implies that 
\begin{equation}\label{asia}
(\1_W T)|_{X'}=\1_{W\cap X'}(T|_{X'}).
\end{equation}
Indeed \eqref{asia} holds for Zariski-closed $W$ in $X$ in view of Proposition \ref{oppna}, and the general case follows by inductively applying (i) and (ii) to $X$ and $X'$. In particular, $(\1_W T)|_{X'}=0$ if $T|_{X'}=0$, and so $\supp \1_W T \subset \supp T$. 
Moreover, $\1_\emptyset T =0$ (by (i)), so if $U=\overline W^c$, then by \eqref{asia} $(\1_W T)|_U=\1_{W\cap U}(T|_U)=\1_\emptyset(T|_U)=0$, and thus $\supp \1_W T \subset \overline W$.
Hence
\begin{equation*}
\supp \1_W T\subset \overline W \cap \supp T.
\end{equation*}
Furthermore, it follows from (i) and (ii) that
\[
\1_{W\cup W'} T=\1_W T + \1_{W'} T - \1_{W\cap W'} T.
\]
Theorem ~\ref{zariski} also implies that
\begin{equation}\label{mannen}
\1_W (\xi\wedge T)=\xi\wedge (\1_W T), \quad \xi \text{ smooth form.}
\end{equation}
Indeed, \eqref{mannen} holds if $W$ is open in light of Proposition ~\ref{oppna} and by (i) and (ii) it extends to all constructible sets.

\begin{proof}
First suppose that $T$ is a sum of elementary currents in $\C^n_\sigma$, that is, $T=\sum \tau_j$, where each $\tau_j$ is of the form \eqref{principe}, and moreover that $W$ is in the Boolean algebra $\B(H_1,\ldots, H_n)$ generated by the coordinate hyperplanes $H_i=\{\sigma_i=0\}$.

The constructible sets in $\mathcal B(H_1,\ldots, H_n)$ can be seen to correspond precisely to subsets of the power set $\P([n])$ of $[n]=\{1,\ldots, n\}$. 
First, identify $\omega\in\P([n])$ with the constructible set 
\[
W_\omega=\{\sigma_i=0 \text{ if } i\in\omega, \sigma_i\neq 0 \text{ if } i\notin \omega\};
\]
then all $W_\omega$ are disjoint and their union is $\C^n$.  
Next, we claim that to each $W\in\B(H_1,\ldots, H_n)$ there is a unique $\Omega=\Omega(W)\subset \P([n])$ such that $W=\bigcup_{\omega\in \Omega} W_\omega$. To see this first note that if such a $\Omega$ exists it is unique since the $W_\omega$ are disjoint. Next, observe that $H_i=\bigcup_{\omega\ni i} W_\omega$ and furthermore that if $\Omega(W)$ and $\Omega(W')$ are well defined, then 
\begin{equation}\label{purple}
(\Omega(W))^c=\Omega(W^c) \text{ and } \Omega(W)\cap \Omega(W')=\Omega(W\cap W')
\end{equation}
and so $\Omega(W^c)$ and $\Omega(W\cap W')$ are well defined. 
The claim now follows by induction, and so each constructible set in $\mathcal B(H_1,\ldots, H_n)$ is represented by an element in $\P([n])$.

Let $\D$ be the mapping from the set of elementary currents on $\C^n_\sigma$ to $\P([n])$ that sends \eqref{principe} to the set $\{i_1,\ldots, i_q\}$ corresponding to its residue factors. Then the mapping 
\begin{equation}\label{simpel}
(W,T)\mapsto \langle W,T \rangle =\sum_{j: \D(\tau_j)\in \Omega(W)} \tau_j 
\end{equation}
is clearly linear in $T$ and moreover $\langle W,T \rangle$ is in $\PM^{p,q}(\C^n)$ if $T$ is. 
Also, because of \eqref{purple},
\begin{equation}\label{appel}
\langle W^c, T \rangle = T - \langle W, T \rangle \text{ and }
\langle W\cap W', T \rangle = \langle W, \langle W', T \rangle \rangle.
\end{equation}

Now, let us fix $W\in\mathcal C(X)$. Then there exist constructible sets $W_1,\ldots, W_s=W$ and Zariski-closed sets $V_1,\ldots V_r$, such that $W_{j+1}$ is of the form $W_{j+1}=A^c$ or $W_{j+1}=B\cap C$, for some $A,B,C\in\{W_1,\ldots,W_j\}\cup\{V_k\}$. Recall that by Proposition \ref{oppna}, $\1_{V_k}T$ is well-defined for all $T\in\PM(X)$. We can therefore define $\1_W$ inductively by letting $\1_{W_{j+1}}T$ equal $(1-\1_A)T$ or $\1_B(\1_C T)$, respectively.

In order to show that this definition is independent of the ``representation'' $\{W_j, V_k\}$ of $W$ let us fix $\{W_j, V_k\}$ and $T\in\PM(X)$. Let $\Pi$ be (compositions of) resolutions of singularities such that $T$ is of the form \eqref{guds} and moreover, for each $k$, $\Pi^{-1}(V_k)$ is a union of hyperplanes $h_1,\ldots,h_n$. Note that this implies that each $W_j\in\mathcal B(h_1,\ldots,h_n)$. 
We claim that
\begin{equation}\label{talle}
\1_W T= \sum \Pi_* \langle \Pi^{-1}(W), \tau_\ell \rangle. 
\end{equation}
To prove this claim observe first that by \eqref{gren} \eqref{talle} holds if $W\in\{V_k\}$. Next, assume that \eqref{talle} holds for $W_1,\ldots,W_j$, and let $A,B,C\in\{W_1,\ldots ,W_j\}\cup\{V_k\}$. Then 
\[
\1_{A^c}T=(1-\1_A)T=\sum \Pi_*(\tau_\ell - \langle \Pi^{-1}(A), \tau_\ell\rangle )
\]
and
\[
\1_{B\cap C}T=\1_B(\1_C)T=\sum \Pi_*\langle \Pi^{-1}(B), \langle \Pi^{-1}(C), \tau_\ell \rangle \rangle.
\]
(To be precise, for the last statement we have used the fact that if \eqref{talle} holds for $T=\sum_{\ell\in L}\Pi_* \tau_\ell$, then it also holds for 
$\sum_{\ell\in L'}\Pi_* \tau_\ell$ if $L'\subset L$.)
By \eqref{appel} it follows that for any elementary current $\tau$, 
\[
1-\langle \Pi^{-1}(A), \tau \rangle = \langle \Pi^{-1}(A)^c, \tau \rangle =
\langle \Pi^{-1}(A^c), \tau \rangle
\]
and
\[
\langle \Pi^{-1}(B), \langle \Pi^{-1}(C), \tau \rangle \rangle =
\langle \Pi^{-1}(B)\cap \Pi^{-1}(C), \tau \rangle =
\langle \Pi^{-1}(B\cap C), \tau \rangle.
\]
Hence \eqref{talle} holds for $W_{j+1}$ and the claim follows by induction. 

Observe that the right hand side of \eqref{talle} only depends of $W$ and not on the representation $\{W_j, V_k\}$. We conlude that the definition of $\1_W$ is intrinsic. 
If we choose $\Pi$ so that $\Pi^{-1}(W)$ and $\Pi^{-1}(W')$ are both unions of hyperplanes it follows from \eqref{appel} and \eqref{talle} that \eqref{kraka} satisfies (i) and (ii). Also, the mapping \eqref{kraka} is linear in $T$ since \eqref{simpel} is. 

%To check that the mapping \eqref{kraka} is linear in $T$, 
%fix $W\in\mathcal C$ and $T, T'\in\PM$ and and choose $\Pi$ so that $T=\sum_{\ell\in L} \Pi_* \tau_\ell$ and $T'= \sum_{\ell\in K} \Pi_* \tau_\ell$, where $\tau_\ell$ are elementary, and moreover $\Pi^{-1}(W)$ is a union of hyperplanes. 
\end{proof}

Observe that a posteriori $\1_W T = \langle W, T \rangle$ if $W\in\mathcal B (H_1, \ldots, H_n)$ and $T$ is a sum of elementary currents. Let us illustrate the mapping \eqref{simpel} with a simple example. 
\begin{ex}\label{dimtva}
Suppose $n=2$. Then the four elements in $\P([2])$, $\{1,2\}$, $\{1\}$, $\{2\}$ and $\emptyset$ correspond to the origin, the $\sigma_2$-axis $H_1$ with the origin removed, the $\sigma_1$-axis $H_2$ with the origin removed, and $\C^2$ with the coordinate axes removed, respectively. For example $H_2$ is given as $W_{\{2\}}\cup W_{\{1,2\}}$. Suppose that 
\[
T=\alpha \left [ \frac{1}{\sigma_1^3} \right ]+
\beta \left [ \frac{1}{\sigma_2} \right ] \dbar \left [ \frac{1}{\sigma_1^2} \right ]+
\gamma ~\dbar \left [ \frac{1}{\sigma_1} \right ]\wedge \dbar \left [ \frac{1}{\sigma_2^2} \right ]=\tau_1+\tau_2+\tau_3,
\]
where $\alpha$, $\beta$ and $\gamma$ are just smooth functions with compact support. Then $\D(\tau_1)=\emptyset$, $\D(\tau_2)=\{1\}$ and $\D(\tau_3)=\{1,2\}$. Now $\langle H_2, T \rangle =\tau_3$ whereas $\langle W, T \rangle =\tau_1+\tau_3$ if $W=W_\emptyset\cup W_{\{1,2\}}$.
\end{ex}
\section{Annihilators of pseudomeromorphic currents}\label{havet}
Let $\PM_x$ denote the $\mathcal E_x$-module of germs of pseudomeromorphic currents at $x\in X$. For $T\in \PM_x$ let $\ann T$ denote the annihilator ideal $\{\varphi\in\O_x; \varphi T=0\}$ in $\O_x$. 

\begin{ex}
Assume $T\in\PM_x$ and let $W$ be a germ of a constructible set at $x$. 
Then if $\varphi\in\O_x$ 
\[
\varphi T=\1_W (\varphi T) + \1_{W^c}(\varphi T)=
\varphi (\1_W T)+\varphi (\1_{W^c} T),
\]
where the first equality follows using (i) and the second one from \eqref{mannen}. Hence if $\varphi\in\ann \1_W T\cap \ann \1_{W^c} T$, then $\varphi\in \ann T$. On the other hand if $\varphi T=0$, then by \eqref{mannen} $\varphi(\1_W T)=\1_W (\varphi T)=0$ and analogously $\varphi (\1_{W^c} T)=0$. Thus 
\[
\ann T=\ann \1_W T\cap \ann \1_{W^c} T.
\]
\end{ex}

For a germ $Z$ at $x$ of a variety let $I_{Z}$ denote the ideal in $\O_x$ of germs of holomorphic functions that vanish on $Z$ and for an ideal $J\subset \O_x$ let $V(J)$ denote the (germ of the) variety of $J$. 

\begin{prop}\label{annanna}
Suppose that $Z$ is an irreducible germ at $x$ of a variety of codimension $q$. If $T\in\PM_x^{p,q}$ has its support in $Z$ then either $T = 0$ or $\ann T$ is an $I_{Z}$-primary ideal. 
\end{prop}

\begin{proof}
Suppose $\varphi\in\O_x$ vanishes on $Z$. Then, since $T$ has finite order, $\varphi^m T =0 $ for $m$ large enough. It follows that $V(\ann T)\subset Z$. 
If $h\in\ann T$, that is, $hT=0$, then $\supp T\subset Z\cap\{h=0\}$. Since $Z$ is irreducible, $Z\cap\{h=0\}$ is either equal to $Z$ or has codimension $\geq q+1$. In the latter case $T=0$ according to Corollary ~\ref{vaga}. Hence $V(\ann T) = Z$ if $T\neq 0$. 

If $\varphi\psi \in\ann T$, then $\varphi\in\ann (\psi T)$. Since $\psi T$ satisfies the assumptions of the proposition, the first part of the proof implies that if $\psi \notin \ann T$, then $\varphi \in I_{Z}=\sqrt{\ann T}$. Thus $\ann T$ is $I_{Z}$-primary. 
\end{proof}

\begin{remark}\label{vid}
Note that the proof only uses that $T$ has the SEP with respect to $Z$. Thus $\ann T$ is $Z$-primary whenever $T\in\PM_x$ has support on $Z$, has the SEP with respect to $Z$ and does not vanish identically.
\end{remark}

\section{Decomposition of $R$ with respect to $\ass J$}\label{stratifikation}
We will now use the results from the the previous sections to make the decomposition of $R$. 
Let us start by briefly recalling the construction of residue currents from ~\cite{AW}. 
Let $J$ be a submodule of $\O_x^{\oplus{r_0}}$, in particular if $r_0=1$, then $J$ is an ideal in $\O_x$, and let 
\begin{equation}\label{upplosning}
0\to \O_x^{\oplus r_N}\stackrel{F_N}{\longrightarrow}\ldots \stackrel{F_2}{\longrightarrow} \O_x^{\oplus r_1}\stackrel{F_1}{\longrightarrow} \O_x^{\oplus r_0},
\end{equation} 
be a free resolution of $\O_x$-modules of $\O_x^{\oplus{r_0}}$, where $J=\Im (\O_x^{\oplus r_1}\to\O_x^{\oplus r_0})$.
Now \eqref{upplosning} induces a holomorphic complex
\begin{equation}\label{vektor}
0\to E_N\stackrel{F_N}{\longrightarrow}\ldots \stackrel{F_2}{\longrightarrow} E_1\stackrel{F_1}{\longrightarrow} E_0,
\end{equation} 
of (trivial) $r_k$-bundles $E_k$ over some neighborhood $\Omega$ of $x\in X$ that is exact outside $Z=V(J)$ and such that $\O_x(E_k)\simeq \O_x^{\oplus r_k}$. Equipping the bundles $E_k$ with Hermitian metrics we construct a current $R$ that has support on $Z$, is annihilated by $\bar I_Z$, and 
\begin{equation}\label{delarna}
R=R_p+\cdots + R_\mu,
\end{equation}
where $p=\codim Z$, $\mu=\min(n,N)$, and $R_j$ is a $(0,j)$-current that takes values in $\Hom (E_0,E_j)$.

Moreover, if $\varphi$ is a germ of a holomorphic section of $E_0$ at $x$, that is, an element in $\O_x^{\oplus r_0}$, then $\varphi\in J$ if and only if $R\varphi =0$ and $\varphi$ lies generically in $\Im F_1$. If $F_1$ is generically surjective, that is, $\codim \O_x^{\oplus r_0}/J > 0$, in particular if $r_0=1$ and $F_1\not\equiv 0$, the latter condition is automatically satisfied and $J=\ann R$. In general, one can extend the complex ~\eqref{vektor} with a mapping $F_0: E_0 \to E_{-1}$ so that the extended complex is generically exact. Then $J=\ker F_0\cap\ann R$. 

Recall that a proper submodule $J$ of the $\O_x$-module $\O_x^{\oplus r}$ is primary if $\varphi\xi\in J$ implies that $\xi \in J$ or $\varphi^\nu\in\ann (\O_x^{\oplus r}/J)$ for some $\nu>0$. If $J\subset\O_x^{\oplus r}$ is primary then $\ann (\O_x^{\oplus r}/J)$ is a primary ideal and so $\p=\sqrt{\ann (\O_x^{\oplus r}/J)}$ is a prime ideal. We say that $J$ is $\p$-primary. As for ideals in $\O_x$, a submodule of $\O_x^{\oplus r}$ always admits a primary decomposition; that is, $J=\bigcap J_k$, where $J_k$ are $\p_k$-primary modules. If all $\p_k$ are different and no intersectands can be removed, then the primary decomposition is said to be minimal and the $\p_k$ are said to be associated prime ideals of $J$. The set of associated prime ideals is unique and we denote it by $\ass J$. 

\begin{ex}\label{kung}
If $F_0: \O_x^{\oplus r_{0}}\to \O_x^{\oplus r_{-1}}$ is a non-zero $\O_x$-homomorphism, then $J=\ker F_0$ is a $\p$-primary module, with $\p=(0)$. Indeed, we have that $\sqrt{\ann \ker F_0}=(0)$, and moreover if $\varphi\in\O_x$ and $\xi\in\O_x^{\oplus r_0}$ are such that $F_0(\varphi \xi)=0$, then $\varphi F_0\xi=0$ and so $\xi\in\ker F_0$ or $\varphi=0$.
\end{ex}

Let $R^{(0)}= F_0$ so that $\ann R^{(0)}=\ker F_0$. For each associated prime ideal $\p\neq (0)$ of $J$ let 
\begin{equation}\label{rp}
R^{\p}=R \1_{V(\p)\setminus \bigcup_{\q\supset\p} V(\q)}.
\end{equation}
In view of \eqref{delarna} (and Corollary ~\ref{vaga}) we have that $R^{\p}=R^{\p}_q+\ldots + R^{\p_i}_\mu$, where $q=\codim \p$ and $R^{\p}_j$ is of bidegree $(0,j)$ and takes values in $\Hom(E_0,E_j)$.  

\begin{thm}\label{prima}
Let $J$ be a submodule of $\O_x^{\oplus r_0}$, and let $R$ be the residue current associated with \eqref{upplosning} (and the choice of Hermitian metrics on the bundles $E_k$ in \eqref{vektor}).
Then for each $\p\in\ass J$, $R^\p$ has support on $V(\p)$ and has the SEP with respect to $V(\p)$, $\ann R^\p\subset \O_x^{\oplus r_0}$ is $\p$-primary, 
\begin{equation}\label{flagga}
R=\sum_{\p\in\ass J,~ \p\neq (0)} R^{\p},
\end{equation} 
and
\begin{equation}\label{primaruppdelning}
J=\ann R^{(0)} \cap \ann R=\bigcap_{\p\in\ass J} \ann R^{\p}
\end{equation}
yields a minimal primary decomposition of $J$.
\end{thm}

The decomposition \eqref{flagga} is unique once the $R^\p$ are required to have support on $V(\p)$ and the SEP with respect to $V(\p)$. Indeed, suppose that $\p$ is of minimal codimension, say $p$, among the associated primes. Then $R^\p=R$ outside a set of codimension $\geq p+1$ and so, because of the SEP, $R^\p$ is uniquely determined. Consequently $R'=\sum_{\codim \p > p} R^\p$, whose support is of codimension $\geq p+1$, is uniquely determined. By the same argument applied to $R'$ all $R^\p$ with $\codim \p=p+1$ are unique. The general statement follows by induction. 
In the same way, as soon as we have the decomposition \eqref{flagga} with the above assumptions on $R^\p$, then \eqref{primaruppdelning} must hold.

\begin{remark}
When constructing the decomposition \eqref{flagga} we have used the a priori knowledge of the associated primes of $J$. However, this is actually not necessary. Suppose that $T\in\PM$ has support on the variety $V$ of pure codimension and let $V_j$ denote the irreducible components of $V$. If $T$ has the SEP with respect to $V$ it follows that $T=\sum T\1_{V_j}$ gives the desired decomposition, that is, $T\1_{V_j}$ has support on $V_j$ and the SEP with respect to $V_j$, and $\ann T=\bigcap \ann T\1_{V_j}$ is a primary decomposition of $\ann T$. If $T$ does not have the SEP with respect to $V$ one can show that there is a subvariety $V'\subset V$ such that $T\1_{V\setminus V'}$ has the SEP with respect to $V$. The above arguments can then be applied to $T\1_{V'}$. After a finite number of steps we obtain a primary decomposition of $\ann T$.  
This idea will be elaborated in more detail in a forthcoming paper.
\end{remark}

We first show a lemma which asserts that $R^\p$ has the SEP. 
\begin{lma}\label{bell}
Suppose that $\p\in \ass J$ is of codimension $q>0$.
Then $\ann R^\p=\ann R_q^\p$. 
Moreover, suppose that $W$ is a variety of codimension
$>q$. Then $R^{\p}\1_W=0$.
\end{lma}

\begin{proof}
Let $Z_k$ denote the set where the mapping $F_k$ in \eqref{vektor} does not have optimal rank. The key observation is that $R_{q+\ell}^\p \1_{Z_{q+\ell}}=0$ for $\ell\geq 1$. To see this let $Z'$ be one of the irreducible components of $Z_{q+\ell}$. If $\codim Z' >q+\ell$, then $R_{q+\ell}^\p \1_{Z'}=0$ due to Corollary ~\ref{vaga}. On the other hand if $\codim Z'=q+\ell$, then $I_{Z'}\in\ass J$ according to Corollary ~20.14 in ~\cite{E}. Thus $R_{q+\ell}^\p\1_{Z'}=R_{q+\ell}^\p\1_{Z'\cap (V(\p)\setminus\bigcup_{\q\supset\p} V(\q))}=0$, since either $I_{Z'}\supset \p$ or $\codim Z'\cap V(\p) > q+ \ell$, in which case the current vanishes according to Corollary ~\ref{vaga}. 

To prove the first statement take $\varphi\in\ann R_q^\p$. Outside $Z_{k+1}$ it holds that $R_{k+1}=\alpha_k R_{k}$, where $\alpha_k$ is a smooth $\Hom(E_k,E_{k+1})$-valued $(0,1)$-form, see for example the proof of Theorem 4.4 in ~\cite{AW}. Now, by ~(i),
\[
R^\p_{q+1}\varphi=(\alpha_q R^\p_{q}\varphi) \1_{X\setminus Z_{q+1}}+(R^\p_{q+1}\varphi) \1_{ Z_{q+1}}=0.
\]
By induction it follows that $R^\p_{q+\ell}\varphi=0$ for $\ell>0$ and so $R^\p\varphi=0$. Thus $\ann R^\p=\ann R^\p_q$. 

For the second statement, note that $R^\p_q\1_W=0$ according to Corollary ~\ref{vaga}. It follows that $R^\p_{q+\ell}\1_W=0$ by the same induction as above. 
\end{proof}

We also need the following module version of Proposition ~\ref{annanna}.

\begin{prop}\label{modulanna}
Suppose that $Z$ is an irreducible germ at $x$ of a variety that has codimension $q$. If $T\in\PM_x^{p,q}(E_0^*)$ 
has its support in $Z$, then either $T \equiv 0$ or $\ann T\subset\O_x(E_0)$ is a $I_Z$-primary module.
\end{prop}

\begin{proof}
For each $\xi \in \O_x(E_0)$, the scalar-valued current $T \xi$ satisfies the assumption of Proposition ~\ref{annanna}. Now, $\ann (\O_x(E_0)/\ann T)=\bigcap_{\xi\in \O_x(E_0)} \ann (T\xi)$. If $T\neq 0$, then $T\xi\neq 0$ for some $\xi\in \O_x(E_0)$ and hence it follows from Proposition ~\ref{annanna} that $V(\ann (\O_x(E_0)/\ann T))=Z$.

Moreover, suppose that $\varphi\in\O_x$ and $\xi\in\O_x(E_0)$ are such that $\varphi \xi\in\ann T$. Since the scalar-valued current $T\xi$ satisfies the assumptions of Proposition ~\ref{annanna} it follows that if $\xi\notin\ann T$, that is, $T\xi\neq 0$, then $\varphi\in I_{Z}$ and thus $\ann T$ is $I_Z$-primary.

\end{proof}

\begin{proof}[Proof of Theorem ~\ref{prima}]
Clearly $R^\p$ has support on $V(\p)$ and Lemma ~\ref{bell} asserts that it has the SEP. 
Throughout this proof we will repeatedly use (i) and (ii) in Theorem \ref{zariski}. 
From Example ~\ref{kung} we know that $\ann R^{(0)}=\ker F_0$ is $(0)$-primary.
Now, suppose that $\p\neq (0)$ and let $q=\codim \p$. Since $R^{\p}_{q}$ is a current of bidegree $(0,q)$ and $V(\p)$ is an irreducible variety of codimension $q$, it follows from Proposition ~\ref{annanna} that $\ann R^{\p}_{q}$ is $\p$-primary. Hence by Lemma ~\ref{bell}, $\ann R^\p$ is $\p$-primary. This could also be seen using Remark ~\ref{vid}.

Next, we show \eqref{flagga}.
By Lemma ~\ref{bell}, for $\p\neq (0)$ we have that $R^\p=R^\p \1_{V(\p)\setminus\bigcup_{\codim \r >\codim \p} V(\r)}$, which by the definition of $R^\p$ is equal to $R \1_{V(\p)\setminus\bigcup_{\codim \r >\codim \p} V(\r)}$.
Suppose that $\p$ and $\q$ are two associated prime ideals of the same codimension $q>0$. 
Then, by Lemma ~\ref{bell}, $R^\p=R^\p \1_{V(\p)\setminus V(\q)}$, since 
$\codim (V(\p)\cap V(\q))>q$. Moreover, in light of \eqref{rp}, this is equal to $R \1_{\left (V(\p)\setminus\bigcup_{\codim \r >k} V(\r)\right) \setminus V(\q)}$. 
Hence, 
\begin{multline*}
R^\p + R^\q =
R\1_{\left (V(\p)\setminus\bigcup_{\codim \r >q} V(\r)\right) \setminus V(\q)}
+R\1_{V(\q)\setminus\bigcup_{\codim \r >q} V(\r)}
=\\
R\1_{(V(\p)\cup V(\q) )\setminus\bigcup_{\codim \r >q} V(\r)},
\end{multline*}
and so
\begin{multline*}
\sum_{\p\in\ass J,~ \p\neq (0)} R^\p =
\sum_{q>0} \sum_{\codim \p = q} R^\p=
\sum_{q>0} R\1_{\bigcup_{\codim \p =q} V(\p)\setminus \bigcup_{\codim \r >q} V(\r)}=\\
R\1_{\bigcup_{\p \in \ass J, ~ \p\neq (0)} V(\p)} = R,
\end{multline*}
since $R$ has support on $V(J)=\bigcup_{\p \in \ass J,~\p\neq (0)} V(\p)$.

We need to show that $\ann R=\bigcap_{\p\in\ass J,~\p\neq (0)}\ann R^\p$. 
Clearly if $R\varphi=0$ then $R^\p\varphi=0$ if $\p\neq (0)$ and so $\ann R\subset \bigcap_{\p\in\ass J,~\p\neq (0)} \ann R^\p$. On the other hand if $R^\p\varphi=0$ for all associated prime ideals $\p\neq (0)$ then by \eqref{flagga} $R\varphi=\sum_{\p\in\ass J, \p\neq (0)} R^\p\varphi=0$ and we are done.
\end{proof}

\begin{ex}
Consider the ideal $(z^2,zw)$ with the associated prime ideals $\p=(z)$ and $\q=(z, w)$, where $\q$ is embedded.
It is easy to see that
\begin{equation*}
0\to \O_x \stackrel{F_2}{\longrightarrow}\O_x^2 \stackrel{F_1}{\longrightarrow} \O_x,
\end{equation*}
where
$
F_1=
\left [ \begin{array}{cc}
z^2 & z w 
\end{array}\right ] 
$
and 
$
F_2=
\left [ \begin{array}{c}
w \\ -z
\end{array}\right ]
$
is (minimal) resolution of $\O_x/J$. 
Assume that the vector bundles in the corresponding complex \eqref{vektor} are equipped with trivial metrics. Then $R^\p=[1/w]\dbar [1/z]$ and $R^\q=\dbar[1/z^2]\wedge\dbar[1/w]$, see Example ~2 in ~\cite{AW} or ~\cite{W2}. Thus we get the minimal primary decomposition 
\[
J=\ann R^\p\cap \ann R^\q = (z)\cap (z^2,w).
\]

Let us also point out that the primary decomposition \eqref{primaruppdelning} in general depends on the choice of Hermitian metrics.
Notice that $J=(z^2, z(w-az))$ for $a\in\C$. 
Thus if we make the same resolution and choice of metrics with respect to the coordinates $\zeta=z, \omega=w-az$, we obtain a residue current that gives the primary decomposition $J=(z)\cap (z^2,w-az),$ which is clearly different for different values of $a$. Now, since all minimal resolutions are isomorphic this new current is obtained from the original resolution with a new choice of metrics.
\end{ex}

\begin{ex}
If $J$ has no embedded primes, it is well known that the minimal primary decomposition is unique. In particular $R^\p$ must be independent of the choice of metrics. This can be verified directly, since outside an exceptional set $\O_x^{\oplus r}/J$ is Cohen-Macaulay and in that case $R$ is essentially canonical, compare to ~\cite{AW}, Section ~4.
\end{ex}

\begin{remark}\label{lolo}[The semi-global case]
Let $K\subset X$ be a Stein compact set, (that is, $K$ admits a neighborhood basis in $X$ consisting of Stein open subsets of $X$) and let $J$ be a submodule of $\O(K)^{r_0}$, where $\O(K)$ is the ring of germs of holomorphic functions on $K$. Due to Proposition 3.3 in ~\cite{AW} $J$ can be represented as the annihilator of a residue current as above. The ring $\O(K)$ is Noetherian precisely when $Z\cap K$ has a finite number of topological components for every analytic variety $Z$ defined in a neighborhood of $K$, see ~\cite{Siu}. In this case the arguments in this and the previous section go through and so we get a decomposition of the residue current analogous to the one in Theorem ~\ref{prima}. 
\end{remark}

\begin{ex}
Let $\J$ be a coherent subsheaf of a locally free analytic sheaf $\O(E_0)$ over a complex manifold $X$. From a locally free resolution of $\O(E_0)/\J$ we constructed in ~\cite{AW} a residue current $R$, whose annihilator sheaf is precisely $\J$. 
Let $Z_k$ be the (intrinsically defined) set where the $k$th mapping in the resolution does not have optimal rank (compare to the proof of Lemma ~\ref{bell}). Then $R$ can be decomposed as $R=\sum_k {}^kR$, where ${}^kR=R \1_{Z_k\setminus Z_{k+1}}$, so that $\J=\bigcap_k \ann {}^kR$ and $\ann {}^kR$ is of pure codimension $k$ (meaning that all its associated primes in each stalk are of codimension $k$). To see this it is enough to show that the germ of ${}^k R$ at $x\in X$ 
satisfies that ${}^kR=\sum_{\codim \p =k} R^\p$, where $\p$ runs over all associated prime ideals of $\J_x$. However, this can be verified following the proofs of Theorem ~\ref{prima} and Lemma ~\ref{bell}.
\end{ex}

\section{The algebraic case}\label{algebraiska}
Let $J$ be a submodule of $\C[z_1,\ldots, z_n]^{r}$ and suppose for simplicity that $\codim \C[z_1,\ldots, z_n]^{r}/J>0$, that is, $(0)\notin \ass J$.  From a free resolution of the corresponding homogeneous modules over the graded ring $\C[z_0,\ldots,z_n]$ we defined in ~\cite{AW} a residue current on $\mathbb P^n$ whose restriction $R$ to $\C_z^n$ satisfies that $\varphi\in \C[z_1,\ldots, z_n]^{r}$ is in $J$ if and only if $R\varphi=0$. Propositions ~\ref{annanna} and ~\ref{modulanna} hold with the same proof if $Z$ is an irreducible algebraic variety in $\C^n$ and $T$ is a current on $\C^n$ of finite order (in particular if it has an extension to $\mathbb P^n$). If we define the currents $R^\p$ for $\p\in\ass J$ as in the local case, the proofs of Lemma ~\ref{bell} and Theorem ~\ref{prima} go through and we obtain the following version of Theorem ~\ref{prima}.

\begin{thm}\label{polynomprima}
Suppose that $J$ is a submodule of $\C[z_1,\ldots, z_n]^{r}$ such that $\C[z_1,\ldots, z_n]^{r}/J$ has positive codimension and let $R$ be a residue current associated with $J$ as above.
Then for each $\p\in\ass J$, $R^\p$ has support on $V(\p)$ and has the SEP with respect to $V(\p)$, $\ann R^\p\subset \C[z_1,\ldots, z_n]^{r}$ is $\p$-primary, 
\begin{equation*}
R=\sum_{\p\in\ass J} R^{\p},
\end{equation*} 
and
\begin{equation*}
J=\ann R=\bigcap_{\p\in\ass J} \ann R^{\p}
\end{equation*}
yields a minimal primary decomposition of $J$.
\end{thm}

In ~\cite{AW} the residue currents for polynomial modules were used to obtain the following version of the Ehrenpreis-Palamodov fundamental principle: any smooth solution to the system of equations 
\begin{equation}\label{diff}
\eta (i \partial/\partial t) \cdot \xi(t)=0, \eta\in J\subset\C[z_1,\ldots, z_n]^{r}
\end{equation}
on a smoothly bounded convex set in $\R^n$ can be written 
\begin{equation*}\label{losning}
\xi(t)=\int_{\C^n}R^T(\zeta) A(\zeta) e^{-i\langle t,\zeta\rangle},
\end{equation*}
for an appropriate explicitly given matrix of smooth functions $A$. Here $R^T$ is (the transpose of) the residue current associated with $J$ as above. Conversely, any $\xi(t)$ given in this way is a homogeneous solution since $J=\ann R$. 
Now, for each $\p\in\ass J$ let 
\begin{equation*}\label{losningsdel}
\xi^\p(t)=\int_{\C^n}(R^\p)^T(\zeta) A(\zeta) e^{-i\langle t,\zeta\rangle},
\end{equation*}
where $R^\p$ is defined above. 
Then by Theorem ~\ref{polynomprima} $\xi=\sum \xi^\p$. Moreover $\xi^\p$ satisfies $\eta (i\partial / \partial t) \cdot \xi^\p=0$ for each $\eta\in\ann R^\p$. Hence we get a decomposition of the space of solutions to \eqref{diff} with respect to $\ass J$. 

\smallskip

\textbf{Acknowledgement:} We would like to thank Ezra Miller for valuable discussions. Thanks also to the referee for pointing out some obscurities and for many helpful suggestions.

\def\listing#1#2#3{{\sc #1}:\ {\it #2},\ #3.}

\end{document}